\documentclass[12pt]{article}

\usepackage{amsmath}
\usepackage{amsthm}
\usepackage{amsfonts}

\usepackage{enumitem}

\usepackage{setspace}

\usepackage{color}

\usepackage{url}

\usepackage{marvosym}
\usepackage{MnSymbol}

\usepackage{verbatimbox} 

\usepackage{tikz-qtree}
\tikzset{downshift/.style={yshift=-40pt}}
\usepackage{cancel}

\usepackage{scalerel}
\usepackage{stackengine,wasysym}

\newcommand\reallywidetilde[1]{\ThisStyle{%
  \setbox0=\hbox{$\SavedStyle#1$}
  \setbox1=\hbox{$x_1$}
  \stackengine{-.1\LMpt}{$\SavedStyle#1$}{%
    \stretchto{\scaleto{\SavedStyle\mkern.2mu\AC}{.5150\wd0}}{.8\ht1}%
  }{O}{c}{F}{T}{S}%
}}

\newcommand{\rev}[1]{{#1}^R}
\newcommand{\drev}[1]{\overset{R}{\reallywidetilde{#1}}}
\newcommand{\revant}[1]{\reallywidetilde{#1}}

\newtheorem{theorem}{Theorem}[section]

\newtheorem{lemma}[theorem]{Lemma}
\newtheorem{corollary}[theorem]{Corollary}

\newcommand{\thistheoremname}{}
\newtheorem*{genericthm}{\thistheoremname}
\newenvironment{namedtheorem}[1]
  {\renewcommand{\thistheoremname}{#1}%
   \begin{genericthm}}
  {\end{genericthm}}

\theoremstyle{definition}
\newtheorem{definition}{Definition}[section]
\newtheorem{conjecture}[theorem]{Conjecture}

\theoremstyle{remark}

\title{On avoidability of formulas with reversal}
\author{James Currie, Lucas Mol, and Narad Rampersad}
\date{}

\begin{document}

\maketitle

%
%
%
%

\begin{abstract}

While a characterization of unavoidable formulas (without reversal) is well-known, little is known about the avoidability of formulas with reversal in general.  In this article, we characterize the unavoidable formulas with reversal that have at most two one-way variables ($x$ is a one-way variable in formula with reversal $\phi$ if exactly one of $x$ and $\rev{x}$ appears in $\phi$).

\vspace{0.5cm}


\end{abstract}


\section{Preliminaries}

An \textit{alphabet} is a finite set of letters.  A \textit{word} $w$ is a finite sequence of letters from some alphabet.  The empty word is denoted $\varepsilon.$  The set of all words over alphabet $A$ (including the empty word) is denoted $A^*.$  For words $v$ and $w,$ we say that $v$ is a \textit{factor} of $w$ if there are words $x$ and $y$ (possibly empty) such that $w=xvy.$  An \textit{$\omega$-word} over alphabet $A$ is an infinite sequence of letters from $A.$  The set of all $\omega$-words over $A$ is denoted $A^\omega.$  A (finite) word $v$ is a \textit{factor} of $\omega$-word $\mathbf{w}$ if there is a word $x$ and an $\omega$-word $\mathbf{y}$ such that $\mathbf{w}=xv\mathbf{y}.$  An $\omega$-word $\mathbf{w}$ is \textit{recurrent} if every finite factor of $\mathbf{w}$ appears infinitely many times in $\mathbf{w}.$

Let $\Sigma$ be a set of letters called \textit{variables}.  A \textit{pattern} $p$ over $\Sigma$ is a finite word over alphabet $\Sigma$.  A \textit{formula} $\phi$ over $\Sigma$ is a finite set of patterns over $\Sigma.$  We usually use dot notation to denote formulas; that is, for $p_1,\dots,p_n\in \Sigma^*$ we let
\[
p_1\cdot p_2\cdot \dots\cdot p_n=\{p_1,p_2,\dots,p_n\}.
\]
Formulas were introduced by Cassaigne \cite{CassaigneThesis}, where it was shown that every formula corresponds in a natural way to a pattern with the same \textit{avoidability index} (see \cite{CassaigneThesis} or \cite{ClarkThesis} for details).  Essentially, this means that formulas are a natural generalization of patterns in the context of avoidability.

For an alphabet $\Sigma,$ define the \textit{reversed alphabet} $\Sigma^R=\{\rev{x}\colon\ x\in \Sigma\},$ where $\rev{x}$ denotes the \textit{reversal} or \textit{mirror image} of variable $x.$  A \textit{pattern with reversal} over $\Sigma$ is a finite word over alphabet $\Sigma\cup\Sigma^R.$  A \textit{formula with reversal} over $\Sigma$ is a finite set of words over $\Sigma\cup\Sigma^R,$ i.e.\ a finite set of patterns with reversal over $\Sigma.$  The elements of a formula (with reversal) $\phi$ are called the \textit{fragments} of $\phi.$  A pattern (with reversal) $p$ is called a \textit{factor} of $\phi$ if $p$ is a factor of some fragment of $\phi.$

For words over any alphabet $A,$ we denote by $\revant{-}$ the reversal antimorphism; if $a_1,a_2,\dots, a_n\in A$, then
\[
\reallywidetilde{a_1a_2\dots a_n}=a_na_{n-1}\dots a_1.
\]
We say that a morphism $f:(\Sigma\cup \Sigma^R)^*\rightarrow A^*$ \textit{respects reversal} if $f(\rev{x})=\revant{f(x)}$ for all variables $x\in \Sigma.$  Note that any morphism $f:\Sigma^*\rightarrow A^*$ extends uniquely to a morphism from $(\Sigma\cup \Sigma^R)^*$ that respects reversal.

Let $p$ be a pattern (with reversal).  An \textit{instance} of $p$ is the image of $p$ under some non-erasing morphism (respecting reversal).  A word $w$ \textit{avoids} $p$ if no factor of $w$ is an instance of $p.$  Let $\phi$ be a binary formula (with reversal).  We say that $\phi$ \textit{occurs} in $w$ if there is a non-erasing morphism $h$ (which respects reversal) such that the $h$-image of every fragment of $\phi$ is a factor of $w.$  In this case we say that $\phi$ occurs in $w$ \textit{through $h$}, or that $w$ \textit{encounters} $\phi$ through $h.$  If $\phi$ does not occur in $w$ then we say that $w$ \textit{avoids} $\phi.$  We say that $\phi$ is \textit{avoidable} if there are infinitely many words over some finite alphabet $A$ which avoid $\phi$. Equivalently, $\phi$ is avoidable if there is a (recurrent) $\omega$-word $\mathbf{w}$ over some finite alphabet $A$ such that every finite prefix of $\mathbf{w}$ avoids $\phi.$ If $\phi$ is not avoidable, we say that $\phi$ is \textit{unavoidable}.

In order to define divisibility of formulas with reversal, we require a different notion of reversal in $(\Sigma\cup\rev{\Sigma})^*$ which not only reverses the letters of a word in $(\Sigma\cup \rev{\Sigma})^*$, but also swaps $x$ with $\rev{x}$ for all $x\in\Sigma.$  For $x_1,x_2,\dots,x_n\in\Sigma\cup \rev{\Sigma},$ we define \textit{$d$-reversal} $\drev{-}$ by
\[
\drev{x_1x_2\dots x_n}=\revant{\rev{x}_1\rev{x}_{2}\dots\rev{x}_n}=\rev{x}_n\rev{x}_{n-1}\dots\rev{x}_1,
\]
where $\rev{(\rev{x})}=x$ for all $x\in\Sigma.$  A morphism $h:(\Sigma\cup\rev{\Sigma})^*\rightarrow(\Sigma\cup \rev{\Sigma})^*$ \textit{respects $d$-reversal} if 
\[
h(\rev{x})=\drev{h(x)}
\] 
for all $x\in\Sigma.$  Note that any morphism $f:\Sigma^*\rightarrow (\Sigma\cup\rev{\Sigma})^*$ extends uniquely to a morphism from $(\Sigma\cup \Sigma^R)^*$ that respects $d$-reversal.

Let $\phi$ and $\psi$ be formulas with reversal.  We say that $\phi$ \textit{divides} $\psi$ if there is a non-erasing morphism $h:(\Sigma\cup \rev{\Sigma})^*\rightarrow (\Sigma\cup \rev{\Sigma})^*$ which respects $d$-reversal such that the $h$-image of every fragment of $\phi$ is a factor of $\psi.$  For example, the formula $xyx\cdot \rev{y}$ divides the formula $xyzxyz\cdot \rev{z}\rev{y}\rev{z}$ through the morphism respecting $d$-reversal $h$ defined by $h(x)=x$ and $h(y)=yz$, meaning 
\[
h(\rev{x})=\drev{h(x)}=\rev{x} \mbox{ and } h(\rev{y})=\drev{h(y)}=\rev{z}\rev{y}.
\] 
This is easily verified as $h(xyx)=xyzx$ is a factor of the fragment $xyzxyz$ and $h(\rev{y})=\rev{z}\rev{y}$ is a factor of the fragment $\rev{z}\rev{y}\rev{z}.$ 

It is straightforward to show that if $\phi$ divides $\psi$ through morphism respecting $d$-reversal $h$ and $\psi$ occurs in a word $w$ through morphism respecting reversal $f$, then $f\circ h$ respects reversal and $\phi$ occurs in $w$ through $f\circ h.$  Thus if $\psi$ is unavoidable and $\phi$ divides $\psi$, then $\phi$ is unavoidable as well.  We say that $\phi$ and $\psi$ are \textit{equivalent} if $\phi$ divides $\psi$ and $\psi$ divides $\phi.$

\begin{definition}
Let $\phi$ be a formula with reversal over alphabet $\Sigma.$  A variable $x\in \Sigma$ is called 
\begin{itemize}
\item \textit{two-way} in $\phi$ if both $x$ and $\rev{x}$ appear in $\phi$;
\item \textit{one-way} in $\phi$ if either $x$ or $\rev{x}$ appears in $\phi$ (but not both); and
\item \textit{absent} from $\phi$ if neither $x$ nor $\rev{x}$ appears in $\phi.$
\end{itemize}
\end{definition}

Note that if $x$ is a one-way variable in $\phi,$ then $\phi$ is equivalent to a formula with reversal in which $x$ appears and $\rev{x}$ does not.  Unless explicitly stated otherwise, when we say that $x$ is a one-way variable in $\phi,$ we will assume that $x$ appears in $\phi$ and $\rev{x}$ does not.  Note also that if $x$ is a two-way variable in $\phi$ and $\phi$ occurs in a word $w$ through morphism respecting reversal $h,$ then $h(x)$ and $h(\rev{x})=\revant{h(x)}$ must both appear in $w$, i.e.\ $h(x)$ is a \textit{reversible factor} of $w.$

Unavoidable formulas without reversal are well understood.  A characterization of unavoidable patterns (which easily generalizes to formulas) was given independently by Bean, Ehrenfeucht, and McNulty \cite{BEM1979} and Zimin \cite{Zimin1984}.  The statement of this result requires some terminology.  

Let $\phi$ be a formula over $\Sigma.$  For each variable $x\in \Sigma,$ make two copies of $x$: $x^\ell$ and $x^r.$  The \textit{adjacency graph} of $\phi,$ denoted $AG(\phi),$ is the bipartite graph on vertex set $\{x^\ell\colon\ x\in \Sigma\}\cup \{x^r\colon\ x\in \Sigma\}$ and edge set $\{\{x^\ell,y^r\}\colon\ xy \mbox{ is a factor of } \phi\}.$  A nonempty subset $F$ of letters appearing in $\phi$ is called a \textit{free set} of $\phi$ if $x^\ell$ and $y^r$ are in different connected components of $AG(\phi)$ for any $x,y\in F.$ 

We say that a formula $\phi$ \textit{reduces} to $\psi$ if $\psi=\delta_F(\phi)$ for some free set $F,$ where $\delta_F(\phi)$ is the formula obtained from $\phi$
by deleting all occurrences of variables from $F$, discarding any empty fragments (denoted $\delta_x(\phi)$ if $F=\{x\}$).  We say that $\phi$ is \textit{reducible} if there is a sequence of formulas $\phi=\phi_0,\phi_1,\dots,\phi_k$ such that $\phi_i$ reduces to $\phi_{i+1}$ for all $i\in\{0,\dots,k-1\}$ and $\phi_k$ is the empty set.

\begin{theorem}[Bean et al.\ \cite{BEM1979} and Zimin \cite{Zimin1984}]
A formula is avoidable if and only if it is reducible.
\end{theorem}

A second useful characterization of unavoidable patterns was proven by Zimin.  Let $x_n$ for $n\in\mathbb{N}$ be different variables.  Let $Z_0=\varepsilon$ and for $n\geq 1$ define $Z_{n}=Z_{n-1}x_nZ_{n-1}.$  The patterns $Z_n$ are called \textit{Zimin words} or \textit{sesquipowers}.  Zimin words are easily seen to be unavoidable, and in fact they are maximally unavoidable in some sense, as every unavoidable formula divides some Zimin word.  The following result stated for patterns in \cite{Zimin1984} is easily seen to generalize to formulas.

\begin{namedtheorem}{Zimin's Theorem}[Zimin \cite{Zimin1984}]
Let $\phi$ be a formula (without reversal) on an alphabet of size $n.$  Then $\phi$ is unavoidable if and only if $\phi$ divides the Zimin word $Z_n.$  Moreover, if $F$ is a free set of $\phi$ such that $\delta_F(\phi)$ is unavoidable, then $\phi$ divides $Z_n$ through a morphism $h$ with $h(y)=x_1$ for every $y\in F.$
\end{namedtheorem}

Little is known about the avoidability of formulas with reversal in general.  In \cite{CurrieLafrance2016}, Currie and Lafrance classified all binary patterns with reversal (i.e.\ patterns with reversal over $\Sigma=\{x,y\}$) by their avoidability index.  In particular, they found that every unavoidable binary pattern with reversal is equivalent to some factor of $xyx$ or $xy\rev{x}.$  In \cite{HighIndex}, the authors presented a family of formulas with reversal of high avoidability index.  In this article, we begin work on a generalization of Zimin's Theorem to formulas with reversal.

\section{Zimin formulas with reversal}

We begin by defining the \textit{Zimin formulas with reversal}, which generalize the Zimin words.  Ideally, we would like to define the Zimin formulas with reversal so that the obvious generalization of Zimin's Theorem holds.  In other words, we would like to be able to say that a formula with reversal is unavoidable if and only if it divides some Zimin formula with reversal.  Our main result is that this characterization holds at least for the formulas with reversal that have at most two one-way variables.

It will be convenient to use the notation $x^\sharp=\{x,\rev{x}\}$ for any variable $x\in \Sigma$.  For sets of words $X$ and $Y,$ we let $XY=\{xy\colon\ x\in X, y\in Y\}.$  For example, 
\[
x^\sharp y^\sharp=\{xy,x\rev{y},\rev{x}y,\rev{x}\rev{y}\}
\]
or in dot notation,
\[
x^\sharp y^\sharp=xy\cdot x\rev{y}\cdot \rev{x}y\cdot \rev{x}\rev{y}
\]
For a single word $w$, we often write $w$ in place of $\{w\}$ when using this notation.  For example,
\[
x^\sharp yx^\sharp=xyx\cdot xy\rev{x} \cdot \rev{x} yx\cdot \rev{x}y\rev{x}
\]

\begin{definition}
For nonnegative integers $m$ and $n,$ define the \textit{Zimin formula with reversal} $Z_{m,n}$ by
\[
Z_{m,0}=x_1^\sharp\dots x_m^\sharp
\]
and
\[
Z_{m,n}=Z_{m,n-1}y_nZ_{m,n-1}.
\]
\end{definition}

Note that $Z_{m,n}$ has $m$ two-way variables $x_1,\dots, x_m$ and $n$ one-way variables $y_1,\dots,y_n.$  Also, $Z_{m,n}$ has $(2^{m})^{(2^n)}$ fragments, each of length $(m+1)2^n-1.$  Note also that when $m=0$ we have $Z_{0,n}=Z_n,$ the usual Zimin word.  We have already mentioned that the usual Zimin words are unavoidable, and we now show that this result generalizes to Zimin formulas with reversal.

\begin{theorem}\label{ZiminUnavoidable}
For any $m,n\geq 0,$ the Zimin formula with reversal $Z_{m,n}$ is unavoidable.
\end{theorem}

\begin{proof}
Let $A$ be an alphabet of size $k.$  We prove the stronger statement that $Z_{m,n}$ occurs in any sufficiently long word $w\in A^*$ under a morphism respecting reversal that sends every fragment of $Z_{m,n}$ to the same factor of $w.$  We proceed by induction on $n.$  For the base case, the formula with reversal $Z_{m,0}$ occurs in any word $w=w_1\dots w_m$ of length $m$ under the morphism respecting reversal $h$ defined by $h(x_i)=w_i.$  Note that every fragment of $Z_{m,0}$ is sent to the same factor $w_1\dots w_m.$  Now consider the formula with reversal $Z_{m,n}.$  By the induction hypothesis, there is some $\ell\in\mathbb{N}$ such that the formula $Z_{m,n-1}$ occurs in any word $v\in A^*$ of length $\ell$ through some morphism respecting reversal that sends every fragment of $Z_{m,n-1}$ to the same word.  Let $w$ be any word of length $k^\ell(\ell+1)+\ell$ over $A.$  We can think of $w$ as the concatenation of $k^\ell+1$ factors of length $\ell$ separated by individual letters.  By the pigeonhole principle, at least one of these factors of length $\ell$ appears twice; let $v$ be such a factor.  Let $h$ be a morphism respecting reversal that shows an occurrence of $Z_{m,n-1}$ in $v$ and maps every fragment of $Z_{m,n-1}$ to the same factor $u$ of $v$.  Then certainly $uxu$ is a factor of $w$ for some $x\neq \varepsilon.$ Extending $h$ by $h(y_n)=x,$ we see that the $h$-image of every fragment of $Z_{m,n}$ is $uxu,$ and thus $h$ gives an occurrence of $Z_{m,n}$ in $w$ that satisfies the required condition.
\end{proof}

It follows from Theorem \ref{ZiminUnavoidable} that if formula with reversal $\phi$ divides a Zimin formula with reversal then $\phi$ is unavoidable.  The remainder of this article is devoted to the question of whether or not the converse of this statement holds.  We know that it holds for formulas with reversal with no two-way variables (i.e.\ formulas without reversal) by Zimin's Theorem.  We demonstrate that it also holds for formulas with reversal with any number of two-way variables and at most two one-way variables.

\section{Avoidable formulas with reversal}

In this section, we prove several lemmas which give sufficient conditions for a formula with reversal to be avoidable.  These will be used extensively in the next section.  We begin by introducing a useful operation on patterns and formulas with reversal.

\begin{definition}
Let $p$ be a pattern with reversal over $\Sigma.$  The \textit{flattening} of $p$, denoted $p^\flat,$ is the image of $p$ under the morphism defined by $x\mapsto x$ and $\rev{x}\mapsto x$ for all $x\in\Sigma$.  We say that $p$ \textit{flattens} to $p^\flat.$

The \textit{flattening} of a formula with reversal $\phi,$ denoted $\phi^\flat,$ is the set of flattenings of all fragments of $\phi,$ i.e.\ $\phi^\flat=\{p^\flat\colon\ p\in\phi\}.$  Again, we say that $\phi$ \textit{flattens} to $\phi^\flat.$
\end{definition}

Essentially, flattening a formula with reversal just involves ignoring the superscript $R$ on any mirror image variables that appear.

We make use of \textit{direct product words} in several of the proofs in this section.  For words $v=v_0v_1\dots$ and $w=w_0w_1\dots$ of the same length (possibly infinite) over alphabets $A_v$ and $A_w$, the \textit{direct product} of $v$ and $w,$ denoted $v\oplus w,$ is the word on alphabet $A_v\times A_w$ defined by
\[
v\oplus w=(v_0,w_0)(v_1,w_1)\dots.
\]
It is sometimes helpful to visualize the ordered pairs as column vectors instead, as below:
\[
v\oplus w=\binom{v_0}{w_0}\binom{v_1}{w_1}\dots.
\]
Clearly if a formula (with reversal) $\phi$ occurs in $v\oplus w$ through morphism $h,$ then $\phi$ also occurs in both $v$ and $w$ by considering the corresponding projection of $h.$

Now we are ready to prove some sufficient conditions for a formula with reversal $\phi$ to be avoidable.  The first such condition is that the related formula $\phi^\flat$ is avoidable.

\begin{lemma}\label{FlatLemma}
Let $\phi$ be a formula with reversal.  If $\phi^\flat$ is avoidable, then $\phi$ is avoidable.
\end{lemma}

\begin{proof}
Suppose that $\phi^\flat$ is avoidable.  Let $\mathbf{w}$ be an $\omega$-word that avoids $\phi^\flat.$  We claim that the word $\mathbf{w}\oplus (123)^\omega$ avoids $\phi.$  Suppose otherwise that $\phi$ occurs in $\mathbf{w}\oplus (123)^\omega$ through morphism $h.$  The only reversible factors of $(123)^\omega$ are single letters, so $h$ must map every two-way variable in $\phi$ to a single letter.  However, then $h(\rev{x})=\revant{h(x)}=h(x)$ for every two-way variable $x$ in $\phi,$ and thus $h$ gives an occurrence of $\phi^\flat$ in $\mathbf{w}\oplus (123)^\omega$.  But then $\phi^\flat$ occurs in $\mathbf{w},$ a contradiction.
\end{proof}

Lemma \ref{FlatLemma} is useful because $\phi^\flat$ has no mirror image variables, and the avoidability of formulas without reversal is well understood.  Using Lemma \ref{FlatLemma}, the following corollaries are easily obtained from well-known sufficient conditions for avoidability of patterns without reversal (Corollary 3.2.10 and Corollary 3.2.11 in \cite{LothaireAlgebraic}, respectively).  

\begin{corollary}
Let $p$ be a pattern with reversal.  If every letter in $p^\flat$ appears twice then $p$ is avoidable.
\end{corollary}

\begin{corollary}\label{Length2tothen}
Let $p$ be a pattern with reversal over an alphabet $\Sigma$ of order $n.$  If $|p|\geq 2^n$ then $p$ is avoidable.
\end{corollary}

The next lemma gives a simple sufficient condition for a formula with reversal with all two-way variables to be avoidable.

\begin{lemma}\label{AllTwoWay}
Let $\phi$ be a formula with reversal such that every variable is two-way in $\phi.$  If some variable appears twice in a single fragment of $\phi^\flat,$ then $\phi$ is avoidable.
\end{lemma}

\begin{proof}
Let $\phi$ be a formula with reversal over an alphabet $\Sigma$ of order $n$ and suppose that some variable $x$ appears twice in some fragment of $\phi^\flat.$  Let $f$ be a minimal factor of $\phi$ that flattens to a factor of $\phi^\flat$ containing two appearances of $x.$  Then $f^\flat=xvx,$ where $v$ is a pattern over $n-1$ variables.  If $|v|\geq 2^{n-1}$ then $\phi$ is avoidable by Corollary \ref{Length2tothen}, so we may assume that $|v|< 2^{n-1}.$  

Let $m=2^{n-1}+1.$  We claim that the word $(123\dots m)^\omega$ avoids $\phi.$  Suppose otherwise that there is a morphism respecting reversal $h$ showing an occurrence of $\phi$ in $(123\dots m)^\omega.$  The only reversible factors in $(123\dots m)^\omega$ are single letters, so $|h(z)|=1$ for all variables $z$ appearing in $\phi.$  But then $h(f)=h(f^\flat)=h(x)h(v)h(x)$ is a factor of $(123\dots m)$ and $|h(v)|=|v|<2^{n-1}.$  So the letter $h(x)$ repeats in $(123\dots m)^\omega$ with at most $2^{n-1}-1=m-2$ letters in between.  This is a contradiction.
\end{proof}

The next lemma concerns formulas with reversal that have at least one one-way variable.  The contrapositive of this lemma is used frequently in the next section: if $\phi$ is unavoidable and has at least one one-way variable, then there is some one-way variable $y$ that appears at most once in any fragment of $\phi.$

\begin{lemma}\label{OneWayTwice}
Let $\phi$ be a formula with reversal with at least one one-way variable.  If for each one-way variable $y$ there is some factor $f_y$ of $\phi$ such that $y$ appears twice in $f_y,$ then $\phi$ is avoidable.
\end{lemma}

\begin{proof}
Let $\phi$ be a formula with reversal over an alphabet $\Sigma$ of order $n$.  Let $\Sigma_1$ be the set of one-way variables in $\phi$ and let $\Sigma_2$ be the set of two-way variables in $\phi.$  For each variable $y\in \Sigma_1,$ assume without loss of generality that $y$ appears in $\phi$ and not $\rev{y}.$  Let $f_y$ be a minimal factor of $\phi$ containing two appearances of $y$ (if there is more than one such factor, choose one).  We have $f_y=yv_yy,$ where $v_y$ is a pattern with reversal over $\Sigma\backslash\{y\}.$  If $|v_y|\geq 2^{n-1}$ then $\phi$ is avoidable by Corollary \ref{Length2tothen}, so we may assume that $|v_y|< 2^{n-1}.$  

Let $\mathbf{w}$ be an $\omega$-word that avoids $(2^{n-1}+1)/2^{n-1}$-powers; such a word exists on $2^{n-1}+2$ letters by Dejean's Theorem (proven independently by Currie and Rampersad \cite{DejeanCurrieRampersad} and Rao \cite{DejeanRao}).  We claim that $\mathbf{w}\oplus (123)^\omega$ avoids $\phi.$  Suppose otherwise that there is a morphism respecting reversal $h$ showing an occurrence of $\phi$ in $\mathbf{w}\oplus(123)^\omega.$  First of all, the only reversible factors in $\mathbf{w}\oplus(123)^\omega$ are single letters, so $|h(z)|=1$ for all $z\in\Sigma_2.$  It follows that $h(z)=h(\rev{z})$ for all $z\in\Sigma_2,$ and hence $h(f_y)=h(f_y^\flat)=h(y)h(v_y)h(y)$ for all $y\in\Sigma_1.$  Let $x\in\Sigma_1$ be a variable satisfying $|h(x)|\geq |h(y)|$ for all $y\in\Sigma_1.$  Then clearly $|h(x)|\geq |h(z)|$ for all variables $z\in\Sigma_1\cup \Sigma_2.$  Since $|v_x|<2^{n-1},$ we see that $h(f_x)=h(x)h(v_x)h(x)$ is at least a $(2^{n-1}+1)/2^{n-1}$-power.  Since $\mathbf{w}$ avoids such powers, $h(f_x)$ is not a factor of $\mathbf{w}\oplus (123)^\omega.$  Therefore, $\phi$ cannot occur in $\mathbf{w}\oplus(123)^\omega.$
\end{proof}

We close this section with one last sufficient condition for a formula with reversal to be avoidable.

\begin{lemma}\label{TwoWayMiddle}
Let $y$ be a two-way variable in a formula with reversal $\phi$ over $\Sigma.$  If $xy, $ $yz$, and $xz$ are factors of $\phi^\flat$ for variables $x,z\in \Sigma$ then $\phi$ is avoidable.
\end{lemma}

\begin{proof}
Let $\phi$ be a formula with reversal with two-way variable $y$ such that $xy,$ $yz,$ and $xz$ are factors of $\phi^\flat.$  We will show that $(123)^\omega$ avoids $\phi.$  Suppose towards a contradiction that $\phi$ occurs in $(123)^\omega$ through morphism $h.$  The only reversible factors of $(123)^\omega$ are single letters, so $|h(y)|=1$ and thus $h(y)=h(\rev{y}).$  Without loss of generality, assume $h(y)=2.$  

If $x$ and $z$ are two-way in $\phi,$ then they are also mapped to single letters by $h$.  If $x$ is one-way in $\phi$, then assume that $x$ appears in $\phi$ (and not $\rev{x}$).  Make the analogous assumption for $z.$  Suppose that factor $f_{xy}$ flattens to $xy,$ $f_{yz}$ flattens to $yz,$ and $f_{xz}$ flattens to $xz.$  If $h(f_{xy})=h(x)h(y)$ is a factor of $(123)^\omega$ then $h(x)$ ends in $1.$  Similarly if $h(f_{yz})=h(y)h(z)$ is a factor of $(123)^\omega$ then $h(z)$ starts with $3.$  But then $h(f_{xz})=h(x)h(z)$ contains the factor $13,$ which does not appear in $(123)^\omega.$  Hence we have reached a contradiction, and $(123)^\omega$ avoids $\phi.$
\end{proof}

In the next section we apply the results from this section as we work towards a characterization of the unavoidable formulas with reversal.

\section{Unavoidable formulas with reversal}

Here we take a first step towards characterizing unavoidable formulas with reversal.  We achieve such a characterization for the formulas with reversal that have at most two one-way variables.

\begin{theorem}\label{MainTheorem}
Let $\phi$ be a formula with reversal with $m\geq 0$ two-way variables and $n\leq 2$ one-way variables.  Then $\phi$ is unavoidable if and only if it divides $Z_{m,n}.$
\end{theorem}

\begin{proof}
First we note that the case $m=0$ is already covered by Zimin's Theorem, so we may assume $m\geq 1.$  By Theorem \ref{ZiminUnavoidable}, $Z_{m,n}$ is unavoidable, so the $(\Leftarrow)$ direction follows immediately.  We now prove the $(\Rightarrow)$ direction.  Let $\phi$ be a formula with reversal with $m\geq 1$ two-way variables $x_1,\dots x_m$ and $n$ one-way variables $y_1,\dots y_n,$ and let $\Sigma=\{x_1,\dots,x_m\}\cup\{y_1,\dots y_n\}.$  Suppose that $\phi$ is unavoidable.  We have two cases: $m=1$ and $m\geq 2.$  Each case is broken into $3$ subcases (for $n=0,1,2,$ respectively).  First we handle the case $m=1.$
\begin{description}
\item[Case 1a) $m=1$, $n=0$:] In this case, $\phi$ is a unary formula with reversal and it is straightforward to show that $\phi$ divides $Z_{1,0}$ using the fact that $\phi^\flat$ must avoid squares (this follows from Lemma \ref{FlatLemma} as $\phi$ is unavoidable).

\item[Case 1b) $m=1,$ $n=1$:] By Lemma \ref{OneWayTwice}, $y_1$ appears at most once in each fragment of $\phi.$  Further, since $\phi^\flat$ avoids squares, every factor of $\phi$ flattens to some factor of $x_1y_1x_1.$  We conclude that $\phi$ divides $Z_{1,1}=x_1^\sharp y_1x_1^\sharp$ through the inclusion map. 

\item[Case 1c) $m=1,$ $n=2$:] By Lemma \ref{OneWayTwice}, some one-way variable (say $y_2$, without loss of generality) appears at most once in each fragment of $\phi.$  We claim that $\{x_1\}$ is a free set in $\phi^\flat$ and that $\delta_{x_1}(\phi^\flat)$ is unavoidable.  First suppose otherwise that $\{x_1\}$ is not a free set of $\phi^\flat.$  Then there is some path from $x_1^\ell$ to $x_1^r$ in the adjacency graph $AG(\phi^\flat)$ of $\phi^\flat.$  Since there are no edges of the form $\{a^\ell,a^r\}$ in $AG(\phi^\flat)$ for $a\in \{x_1,y_1,y_2\},$ the path from $x_1^\ell$ to $x_1^r$ must be $x_1^\ell y_1^r y_2^\ell x_1^r$ or $x_1^\ell y_2^r y_1^\ell x_1^r.$  So $\phi^\flat$ contains the factors $x_1y_1,$ $y_2y_1,$ and $y_2x_1$; or $x_1y_2,$ $y_1y_2,$ and $y_1x_1$; respectively.  Both situations are impossible by Lemma \ref{TwoWayMiddle}.  It remains to show that $\delta_{x_1}(\phi^\flat)$ is unavoidable.  We will show that every fragment of $\delta_{x_1}(\phi^\flat)$ is a factor of $y_1y_2y_1$; it follows that $\delta_{x_1}(\phi^\flat)$ divides $Z_2$ and hence is unavoidable by Zimin's Theorem.  Recall that $y_2$ appears at most once in every fragment of $\phi,$ so it suffices to show that if $y_1$ appears twice in some fragment of $\phi,$ they are on opposite sides of an appearance of $y_2.$  If we replace every occurrence of $y_2$ in $\phi$ with a dot, we are left with a formula with reversal on $\{x_1,y_1\}$ that must be unavoidable.  By Lemma \ref{OneWayTwice}, $y_1$ appears at most once in every fragment of this formula, and thus every fragment of $\delta_{x_1}(\phi^\flat)$ is a factor of $y_1y_2y_1.$

\end{description}

Now we consider the case $m\geq 2.$  Since $\phi$ is unavoidable, it certainly occurs in the word $\mathbf{w}_{m+1}=(x_1\dots x_{m+1})^\omega$ through some morphism respecting reversal $h.$  Since $m\geq 2,$ the $h$-image of every two-way variable is a single letter.  Further, since there are exactly $m$ two-way variables, there is some letter in $\mathbf{w}_{m+1}$ that is not the $h$-image of any two-way variable; we may assume that $x_{m+1}$ is such a letter without loss of generality.

Consider the image $h(y)$ of a one-way variable $y.$  If $|h(y)|>m+1,$ we can remove $m+1$ letters from the end of $h(y)$ without changing the fact that $h$ shows an occurrence of $\phi$ in $w_{m+1}.$  We can also add $m+1$ letters to the end of $h(y)$ by adding a single period of $w_{m+1}$ to $h(y),$ starting at the letter following the last letter of $h(y)$.  Thus we can assume that the letter $x_{m+1}$ appears in $h(y)$ exactly once: if $x_{m+1}$ appears more than once, remove $m+1$ letters from $h(y)$ recursively until $x_{m+1}$ appears exactly once; and if $x_{m+1}$ never appears in $h(y)$, add an appropriate period of $w_{m+1}$ to $h(y).$  From here, we have three subcases as before.

\begin{description}
\item[Case 2a) $m\geq 2,$ $n=0$:] By the observations made above, for any fragment $f$ of $\phi,$ the image $h(f)$ must be a factor of $x_1\dots x_m.$  Define a morphism $\overline{h}:(\Sigma\cup \rev{\Sigma})^*\rightarrow (\Sigma\cup \rev{\Sigma})^*$ respecting $d$-reversal by $\overline{h}(x)=h(x)$ for all $x\in \Sigma$, which means
\[
\overline{h}(\rev{x})=\drev{h(x)}=\rev{h(x)}
\]
for all $x\in\Sigma$ (the last equality follows from the fact that $h(x)$ is a single letter from $\{x_1,\dots,x_m\}$).  In other words, $\overline{h}$ takes the images of $h$ on the letters of $\Sigma$ but extends to a morphism respecting $d$-reversal instead of a morphism respecting reversal.  Thus for any fragment $f$ of $\phi,$ we have $\overline{h}(f)^\flat=h(f).$  Further, since $h(f)$ is a factor of $x_1\dots x_m,$ it follows that $\overline{h}(f)$ is a factor of $Z_{m,0}=x_1^\sharp\dots x_m^\sharp.$  We conclude that $\phi$ divides $Z_{m,0}.$

\item[Case 2b) $m\geq 2$, $n=1$:]  By the observations made above, we may assume that $x_{m+1}$ is not in the $h$-image of any two-way variable, and that $x_{m+1}$ appears exactly once in $h(y_1)$.  By Lemma \ref{OneWayTwice}, $y_1$ appears at most once in any fragment $f$ of $\phi$; hence $h(f)$ must be a factor of $x_1\dots x_mx_{m+1}x_1\dots x_m.$  Define a morphism $\overline{h}$ respecting $d$-reversal by 
\begin{align*}
\overline{h}(x)&=h(x) \mbox{ for all } x\in \Sigma\backslash y_1, \mbox{ and}\\
\overline{h}(y_1)&=t_{y_1}(h(y_1)),
\end{align*}
where $t_{y_1}:\{x_1,\dots, x_{m+1}\}\rightarrow \Sigma$ is defined by
\[
t_{y_1}(x_i)=\begin{cases}
y_1 &\mbox{ if } i=m+1;\\
x_i &\mbox{ otherwise.} 
\end{cases}
\]
Put simply, $t_{y_1}$ swaps $x_{m+1}$ for $y_1$ and leaves all other letters alone.  For any fragment $f$ of $\phi,$ we see that $\overline{h}(f)^\flat$ is a factor of $x_1\dots x_my_1x_1\dots x_m$.  Finally, since $y_1^R$ does not appear in $\phi,$ and hence does not appear in $\overline{h}(f)$ either, we conclude that $\overline{h}(f)$ is a factor of $Z_{m,1}=x_1^\sharp\dots x_m^\sharp y_1 x_1^\sharp\dots x_m^\sharp.$

\item[Case 2c) $m\geq 2,$ $n=2$:]  By the observations made above, we may assume that $x_{m+1}$ is not in the image of any two-way variable, and that $x_{m+1}$ appears exactly once in $h(y_j)$ for $j\in\{1,2\}$.  By Lemma \ref{OneWayTwice}, some one-way variable, say $y_2$, appears at most once in any fragment of $\phi.$  If we replace every appearance of $y_2$ in $\phi$ with a dot, the resulting formula must be unavoidable, and thus by Lemma \ref{OneWayTwice} again we see that the other one-way variable $y_1$ appears at most once in each fragment of the resulting formula.  Thus the variable $y_1$ appears at most twice in any fragment $f$ of $\phi,$ and if it appears twice then $y_2$ is in between the two appearances.  Since $x_{m+1}$ appears only in the images of $y_1$ and $y_2,$ and exactly once in each image, it follows that $h(f)$ must be a factor of $(x_1\dots x_mx_{m+1})^3x_1\dots x_m$ for any fragment $f$ of $\phi.$  Define a morphism $\overline{h}$ respecting $d$-reversal by 
\begin{align*}
\overline{h}(x)&=h(x) \mbox{ for all } x\in \Sigma\backslash \{y_1,y_2\}, \mbox{ and}\\ 
\overline{h}(y_j)&=t_{y_j}(h(y_j)) \mbox{ for } j\in\{1,2\} 
\end{align*}
where $t_{y_j}:\{x_1,\dots, x_{m+1}\}\rightarrow \Sigma$ is a morphism defined by
\[
t_{y_j}(x_i)=\begin{cases}
y_j &\mbox{ if } i=m+1;\\
x_i &\mbox{ otherwise.} 
\end{cases}
\]
Put simply, $t_{y_j}$ swaps $x_{m+1}$ for $y_j$ and leaves all other letters alone.  Now for any fragment $f$ of $\phi,$ we see that $\overline{h}(f)^\flat$ is a factor of 
\[
x_1\dots x_my_1x_1\dots x_my_2x_1\dots x_my_1x_1\dots x_m.
\]  
Finally, since $\rev{y_1}$ and $\rev{y_2}$ do not appear in $\phi,$ and hence do not appear in $\overline{h}(f)$ either, we conclude that $\overline{h}(f)$ is a factor of 
\[
Z_{m,2}=x_1^\sharp\dots x_m^\sharp y_1 x_1^\sharp\dots x_m^\sharp y_2 x_1^\sharp\dots x_m^\sharp y_1 x_1^\sharp\dots x_m^\sharp.
\]
\end{description}
This completes the proof.
\end{proof}

\section{Conclusion}

We have shown that a formula with reversal with at most two one-way variables is unavoidable if and only if it divides a Zimin formula with reversal.  We conjecture that this result generalizes to formulas with reversal with any number of one-way variables.

\begin{conjecture}\label{MainConjecture}
Let $\phi$ be a formula with reversal with $m$ two-way variables and $n$ one-way variables.  Then $\phi$ is unavoidable if and only if it divides $Z_{m,n}.$
\end{conjecture}

We briefly discuss possible approaches to proving Conjecture \ref{MainConjecture}.  In Cases 1c) and 2c) of Theorem \ref{MainTheorem}, we see that when all two-way variables are deleted from an unavoidable formula with reversal with two one-way variables, we are left with a \textit{factor} of $Z_2=y_1y_2y_1.$  This fact does not generalize to the case that we have $n\geq 3$ one-way variables (consider the unavoidable formula $x^\sharp y_1x^\sharp y_2x^\sharp y_3x^\sharp y_1x^\sharp y_2x^\sharp$, for example).  However, it seems plausible that when we delete all two-way variables from an unavoidable formula with reversal we are left with an unavoidable formula.  We state this as a conjecture below.

\begin{conjecture}\label{HelperConjecture}
Let $\phi$ be a formula with reversal with set $X$ of two-way variables.  If $\phi$ is unavoidable then $\delta_{X\cup \rev{X}}(\phi)$ is unavoidable.
\end{conjecture}

Let $\phi$ be a formula with reversal with $n$ one-way variables and set $X$ of two-way variables.  If Conjecture \ref{HelperConjecture} is true, then $\delta_{X\cup \rev{X}}(\phi)$ divides the Zimin word $Z_n$.  This division map (call it $d$) would tell us how to adjust the morphism $h$ through which $\phi$ occurs in $\mathbf{w}_{m+1}=(x_1\dots x_{m+1})^\omega$ to an appropriate $\overline{h}$ as in Theorem \ref{MainTheorem} Case 2c).  This would prove Conjecture \ref{MainConjecture} in the case that $m>1$, so we describe the process now.  For each one-way variable $y,$ we first adjust $h(y)$ to have $|d(y)|$ appearances of $x_{m+1}$ by removing or adding a multiple of $m+1$ letters from $h(y).$  Then we define $\overline{h}$ as in Theorem \ref{MainTheorem} Case 2c), except $t_y$ sends each appearance of $x_{m+1}$ in $h(y)$ to the corresponding letter of $d(y).$ 

\bibliographystyle{amsplain}
\bibliography{Words}

\providecommand{\bysame}{\leavevmode\hbox to3em{\hrulefill}\thinspace}
\providecommand{\MR}{\relax\ifhmode\unskip\space\fi MR }
\providecommand{\MRhref}[2]{%
  \href{http://www.ams.org/mathscinet-getitem?mr=#1}{#2}
}
\providecommand{\href}[2]{#2}
\begin{thebibliography}{1}

\bibitem{BEM1979}
D.~R. Bean, A.~Ehrenfeucht, and G.~F. Mc{N}ulty, \emph{Avoidable patterns in
  strings of symbols}, Pacific J. Math. \textbf{85} (1979), 261--294.

\bibitem{CassaigneThesis}
J.~Cassaigne, \emph{Motifs \'evitables et r\'egularit\'e dans les mots}, Ph.D.
  thesis, Universit\'e Paris VI, 1994.

\bibitem{ClarkThesis}
R.~J. Clark, \emph{Avoidable formulas in combinatorics on words}, Ph.D. thesis,
  University of California, Los Angeles, 2001.

\bibitem{CurrieLafrance2016}
J.~D. Currie and P.~Lafrance, \emph{Avoidability index for binary patterns with
  reversal}, Electron. J. Combin. \textbf{23} (2016), no.~1.

\bibitem{HighIndex}
J.~D. Currie, L.~Mol, and N.~Rampersad, \emph{A family of formulas with
  reversal of high avoidability index}, preprint, arXiv:1611.03535[math.CO],
  2016.

\bibitem{DejeanCurrieRampersad}
J.~D. Currie and N.~Rampersad, \emph{A proof of {D}ejean's conjecture}, Math.
  Comp. \textbf{80} (2011), no.~274, 1063--1070.

\bibitem{LothaireAlgebraic}
M.~Lothaire, \emph{Algebraic combinatorics on words}, Cambridge University
  Press, 2002.

\bibitem{DejeanRao}
M.~Rao, \emph{Last cases of {D}ejean's conjecture}, Theoret. Comput. Sci.
  \textbf{412} (2011), no.~27, 3010--3018.

\bibitem{Zimin1984}
A.~I. Zimin, \emph{Blocking sets of terms ({E}nglish translation)}, Math. USSR
  Sbornik \textbf{47} (1984), no.~2, 353--364.

\end{thebibliography}

\end{document}